\documentclass[12pt]{amsart}

\usepackage{amssymb}

\usepackage{enumitem}

\usepackage{graphicx}

\makeatletter
\@namedef{subjclassname@2020}{%
  \textup{2020} Mathematics Subject Classification}
\makeatother

\usepackage[T1]{fontenc}

\newtheorem{theorem}{Theorem}[section]
\newtheorem{corollary}[theorem]{Corollary}

\newtheorem{proposition}[theorem]{Proposition}

\theoremstyle{definition}
\newtheorem{definition}[theorem]{Definition}
\newtheorem{remark}[theorem]{Remark}
\newtheorem{example}[theorem]{Example}

\numberwithin{equation}{section}

\newcommand{\RR}{\mathbb{R}}
\newcommand{\CC}{\mathbb{C}}

\newcommand{\tor}{\operatorname{Tor}}

\newcommand{\aba}{\bar{\alpha}}
\newcommand{\bba}{\bar{\beta}}

\newcommand{\gba}{\bar{\gamma}}

\newcommand{\sba}{\bar{\sigma}}
\newcommand{\wba}{\bar{w}}
\newcommand{\Wba}{\overline{W}}
\newcommand{\zba}{\bar{z}}
\newcommand{\Zba}{\overline{Z}}
\newcommand{\II}{\mathrm{I\!I}}

\newcommand{\jba}{\bar{j}}
\newcommand{\kba}{\bar{k}}

\newcommand{\vr}{\varrho}
\newcommand{\tf}{\operatorname{tf}}

\newcommand{\edit}[1]{{#1}}

\begin{document}




\title[The holomorphic sectional curvature]{The holomorphic sectional curvature\\ and ``convex'' real hypersurfaces\\ in K\"ahler manifolds}

\author{Duong Ngoc Son}
\address{Fakultät für Mathematik, Universität Wien, Oskar-Morgenstern-Platz~1, 1090 Wien, Austria}
\email{son.duong@univie.ac.at}
\date{April 4, 2021}

\begin{abstract}
	We prove a sharp lower bound for the Tanaka-Webster holomorphic sectional curvature of strictly pseudoconvex real hypersurfaces that are ``semi-isometrically'' immersed in a K\"ahler manifold of nonnegative holomorphic sectional curvature under an appropriate convexity condition. This gives a partial answer to a question posed by Chanillo, Chiu, and Yang regarding the positivity of the Tanaka-Webster scalar curvature of the boundary of a strictly convex domain in~$\CC^2$ from 2012. In fact, the main result proves a stronger positivity property, namely the $\frac12$-positivity in the sense of Cao, Chang, and Chen, for compact ``convex'' real hypersurfaces in a K\"ahler manifold of nonnegative holomorphic sectional curvature. Our approach is rather simple and uses a version of the Gauss equation for semi-isometric CR immersions of pseudohermitian manifolds into K\"ahler manifolds.
\end{abstract}

\subjclass[2000]{32V20, 53B99}
\thanks{The author was supported by the Austrian Science Fund FWF-project M 2472-N35.}

\keywords{holomorphic sectional curvature, convexity}

\maketitle
\section{Introduction}

Let $(M,\theta)$ be a pseudohermitian manifold and let $R_{\alpha\bba\gamma\sba}$ be the Tanaka-Webster curvature tensor of $\theta$ \cite{tanaka1975differential,webster1978pseudo}. The holomorphic sectional curvature $K(Z)$, the Ricci tensor $R_{\alpha\bba}$, and the scalar curvature $R$ can be defined as usual. For example, Webster \cite{webster1978pseudo} defined the {holomorphic sectional curvature} of $(M,\theta)$ for each non-zero $(1,0)$-vector $Z = \zeta^{\alpha} Z_{\alpha}$ by
\begin{equation}\label{e:secdef}
	K(Z)
	:=
	\frac{1}{2} R_{\alpha\bba\gamma\sba} \zeta^{\alpha} \zeta^{\bba} \zeta^{\gamma} \zeta^{\sba} /|Z|^4.
\end{equation}
These curvature quantities, essentially coincide when $M$ is of 3-dimensional, are fundamental in the study of pseudohermitian and CR geometry of Levi-nondegenerate real hypersurfaces (or more generally, Levi-nondegerate CR manifolds of hypersurface type) and have been investigated extensively in the literature. For example, in \cite{tanaka1975differential}, Tanaka proved in the 1970s that for compact manifolds if the Tanaka-Webster Ricci tensor is positive then the first Kohn-Rossi group $H^{0,1}(M)$ is trivial (Tanaka's statement is actually more general, see \cite[Proposition 5.3]{tanaka1975differential} for details). In the 2010s, Chanillo, Chiu, and Yang \cite{chanillo2012embeddability} proved that for a compact three-dimensional stricly pseudoconvex CR manifold, the positivity of the scalar curvature and the nonnegativity of the CR Paneitz operator imply the embeddability of it into a complex space. This result partially motivated further research on the positivity of the scalar curvature and the CR Paneitz operator, see \cite{chanillo2012embedded}. Chanillo, Chiu, and Yang proved in \cite{chanillo2012embedded}, among other interesting results, that the scalar curvature $R$ is positive on a real ellipsoid of $\CC^2$ \cite[Theorem~1.2]{chanillo2012embedded}. In general dimension, a positive lower bound for the Tanaka-Webseter Ricci tensor on a real ellipsoid can be deduced from Li-Tran \cite{li2011cr} which studies a lower estimate for the eigenvalue of the sub-Laplacian \cite[Theorem~1.1]{li2011cr}. In \cite[page 81]{chanillo2012embedded}, Chanillo, Chiu, and Yang posed an interesting question asking if the scalar curvature $R$ must be positive on strictly convex domains in $\CC^2$?

In this paper, we give a partial affirmative answer to this question for a special class of real hypersurfaces in $\CC^2$. More precisely, for the ``pluriharmonic perturbations'' of the sphere, the $\CC$-convexity (a weaker notion of the usual convexity, see, e.g., \cite{andersson2012complex}) does imply a positive lower bound for the holomorphic sectional curvature $K(Z)$. This is a special case of a more general result for pseudohermitian manifolds that are ``convex'' (in an appropriate sense) ``semi-isometrically'' immersed (in the sense of \cite{son2019semi}) in a K\"ahler manifold having nonnegative holomorphic sectional curvature. Here, the new convexity is defined as follows. First, recall that a pseudohermitian manifold $(M,\theta) \subset (\mathcal{X},\omega)$ is said to be \textit{semi-isometrically immersed} in a K\"ahler manifold $(\mathcal{X},\omega)$ if $d\theta = \iota^{\ast} \omega$ where $\iota$ is the inclusion. The second fundamental form $\II$ for $\iota$ is defined using the Chern connection $\widetilde{\nabla}$ on $\mathcal{X}$ and the Tanaka-Webster connection $\nabla$ on $M$ via a Gauss formula, that is,
\begin{align}
	\II(X,Y) := \widetilde{\nabla}_XY - \nabla_XY,
\end{align}
for vector fields $X,Y \in \CC T(M)$, extended smoothly to vector fields in $\CC T(\mathcal{X})$. If $Z_{\alpha}$ is a local unitary frame for $T^{(1,0)}(M)$, then the $(1,0)$-mean curvature is defined \cite{son2019semi} by
\begin{align}
	H := \sum_{\alpha =1}^n \II(Z_{\aba},Z_{\alpha}).
\end{align}
Then, $H$ is a section of the normal bundle $N^{(1,0)}(M)$, the orthogonal complement of $T^{(1,0)}(M)$ in $T^{(1,0)}(\mathcal{X})$ \cite{son2019semi}. We also define the real mean curvature vector by
\begin{align}
	\mu = \frac{1}{2}(H+\overline{H}).
\end{align}
If $M$ is a real hypersurface in $\mathcal{X}$, then the $(1,0)$-normal bundle $N^{(1,0)}(M)$ has 1-dimensional fibers and the restriction $\II|_{H(M)}$, $H(M):=\Re T^{(1,0)}(M)$, is determined by a scalar-valued form
\begin{align}\label{e5}
	h(X,Y) = \langle \II(X,Y), \mu\rangle /|\mu|, \quad X,Y \in H(M).
\end{align}
We can state the definition of the $\CC$-convexity that is relevant for our purposes as follows.
\begin{definition}\label{def1}
	Let $(M,\theta) \subset (\mathcal{X},\omega)$ be a strictly pseudoconvex semi-isometrically immersed real hypersurface in a K\"ahler manifold. We say that $(M,\theta)$ is \emph{pseudohermitian $\CC$-convex} if the scalar second fundamental form $h$ is positive semi-definite in the complex tangent space $H(M)$.
\end{definition}
Clearly, this definition is analogous to the well-known notion of $\CC$-convexity for subsets of $\CC^{n+1}$ or $\mathbb{CP}^{n+1}$ in the complex analysis literature; see \cite{andersson2012complex} and Section~\ref{subsec:convex}.

In analogy with the notion of holomorphic sectional curvature \eqref{e:secdef}, we define the \textit{holomorphic sectional torsion} $A(Z)$ for each non-zero vector $Z = \zeta^{\alpha} Z_{\alpha}$ in $T^{(1,0)}(M)$ by
\begin{align}
	A(Z): = iA_{\alpha\beta} \zeta^{\alpha}\zeta^{\beta}/|Z|^2,
\end{align}
where $A_{\alpha\beta}$ is the pseudohermitian torsion; see Section 2 for details.
It should be mentioned that twice the real part of $|Z|^2 A(Z)$, often denoted by $\mathrm{Tor}(Z,Z)$, i.e.
\begin{align*}
	\mathrm{Tor}(Z,Z) :=
	iA_{\alpha\beta} \zeta^{\alpha}\zeta^{\beta} - iA_{\aba\bba} \zeta^{\aba}\zeta^{\bba},
\end{align*}
plays an important role in several problems in the pseudohermitian geometry. We mention here, for examples, the Lichnerowicz-type estimate for the first positive eigenvalue of the sub-Laplacian (see e.g. \cite{li2011cr} and the references therein) and the classification of the closed CR torsion solitons~\cite{cao2020}. We have the following theorem.
\begin{theorem}\label{thm:m1}
	Let $(M,\theta)$ be a pseudohermitian real hypersurface of a K\"ahler manifold $(\mathcal{X},\omega)$ with $d\theta =\iota^{\ast} \omega$. Then $(M,\theta)$ is pseudohermitian $\CC$-convex if and only if 
	\begin{align}\label{e:mesta}
		|A(Z)| \leqslant |H|^2
	\end{align}
	for all $Z \in T^{(1,0)}(M)$, $Z\ne 0$. In this case the Tanaka-Webster holomorphic sectional curvature $K$ of $(M,\theta)$ and the holomorphic sectional curvature $\tilde{K}$ of $(\mathcal{X},\omega)$ satisfy
	\begin{equation}\label{e:mest}
		K(Z) \geqslant \frac{\tilde{K}(Z) + |H|^2}{2}.
	\end{equation}
	Moreover, the inequalities are sharp.
\end{theorem}
The difference between the coefficients of $K$ and $\tilde{K}$ \edit{in \eqref{e:mest}} is artificial and due to the factor $\frac{1}{2}$ in \eqref{e:secdef}, which was chosen by Webster \cite{webster1978pseudo} so that the hypersphere $|z|^2 =1$ with the ``standard'' pseudohermitian structure $\theta: = i\bar\partial |z|^2$ has constant curvature $+1$.

The arguably simplest case in which Theorem~\ref{thm:m1} applies is when $\mathcal{X}$ is a K\"ahler surface of nonnegative holomorphic sectional curvature. In this case, $(M,\theta)$ is of three-dimensional and its pseudohermitian $\CC$-convexity implies the positivity of the Tanaka-Webster scalar curvature. In analogy with the positive scalar curvature problem in Riemannian geometry, we can ask what CR manifolds can be endowed with a positive Tanaka-Webster scalar curvature pseudohermitian structure? In certain sense the existence of such a structure is weaker than its Riemannian analogue. In fact, there are CR manifolds diffeomorphic to 3-torus $\mathbb{T}^3$ which admit pseudohermitian structures of positive Tanaka-Webster scalar curvature (see Remark 3.4) while tori of dimension at least 3 admit \textit{no} positive scalar curvature Riemannian metric \cite{schoen}. Recently, Cao, Chang, and Chen \cite{cao2020} considered a condition stronger than the positivity of the Tanaka-Webster scalar curvature, the $C_0$-positivity, which involves both curvature and torsion. For compact CR 3-manifolds they proved that the $C_0$-positivity with $C_0 \geqslant \frac12$ implies the positivity of an adapted Riemannian metric. Moreover, the threshold $\frac12$ is sharp (see Remark 3.4).

Combining this result and Theorem~\ref{thm:m1}, we obtain the following
\begin{corollary}\label{cor13a}
	Let $(M^3,\theta)$ be a compact three-dimensional strictly pseudohermitian $\CC$-convex real hypersurface in a K\"ahler surface $(\mathcal{X}, \omega)$ with nonnegative holomorphic sectional curvature. Then $M$ admits a Riemannian metric of positive curvature.
\end{corollary}
We point out that both the convexity of $(M^3,\theta)$ and the nonnegativity of the sectional curvature of the ambient space are necessary for Corollary~\ref{cor13a}. Indeed, consider the CR link $\Sigma(p,q,r)$ of a Brieskorn-Pham singularity $V(p,q,r): = \{z_1^p + z_2^q + z_3^r = 0\} \subset \CC^3$. Generally, $V(p,q,r)$ with the induced metric from a suitable flat metric on $\CC^3$ has nonpositive holomorphic sectional curvature. For an appropriate choice of contact form $\theta$, the link $\Sigma(p,q,r)$ is a pseudohermitian $\CC$-convex (in fact, of vanishing pseudohermitian torsion) real hypersurface in $V(p,q,r)$. On the other hand, it is well-known that many of these CR links admit no positive scalar curvature metrics \cite{gl,milnor}, see also \cite{ohta}. Thus, we obtain a wealth of examples which show that the nonnegativity of $\tilde{K}$ in Corollary~\ref{cor13a} is necessary. See Section~\ref{sec:brieskorn} for details.

In the rest of this introduction, we describe a concrete and familiar situation in which the pseudohermitian $\CC$-convexity in the sense of Definition~\ref{def1} coincides with the well-known and more intuitive notion of \textit{$\CC$-convexity} in complex analysis as defined in \cite{andersson2012complex}. The latter, in turns, is weaker than the ``usual'' convexity for subsets of the complex euclidean space. We thus obtain a family of manifolds for which the pseudohermitian $\CC$-convexity is evident. More precisely, let $M\subset \CC^{n+1}$ be a real hypersurface defined by $\vr = 0$ with $d\vr \ne 0$. Recall that if the restriction of the Hessian of $\vr$ onto the complex tangent space $H(M)$ is positive semi-definite, then $M$ is $\CC$-convex \cite{andersson2012complex}. 

In what follows, we shall assume the following condition holds along $M$:
\begin{align}\label{e:10}
	\frac{\partial^2 \vr}{\partial z_j \partial \zba_{k}} = a_{j\kba} + o(\vr)
\end{align}
where $a_{j\kba}$ is a positive definite Hermitian matrix with constant entries. Then $M$ is strictly pseudoconvex. Let $\theta = \iota^{\ast}(i\bar\partial \vr)$ be a pseudohermitian structure on $M$ and $\omega = ia_{j\kba} dz^j \wedge d\zba^k$ the K\"ahler form of a flat K\"ahler metric on $\CC^{n+1}$. Clearly,
\begin{align*}
	d\theta = \iota ^{\ast} \omega.
\end{align*}
Thus, $\iota$ is semi-isometric. In Proposition~\ref{prop:cconvex} we prove that under condition \eqref{e:10} $M$ is pseudohermitian $\CC$-convex if and only if $M$ is $\CC$-convex in the sense of \cite{andersson2012complex}. 
Moreover, let
\begin{align}
	|\partial\vr|^2 := \sum_{j,k=1}^{n+1} a^{j\kba} \vr_j \vr_{\kba}
\end{align}
be the squared-length of $\partial\vr$ in the (dual) metric $\omega$. Then $|H|^2 = 1/|\partial\vr|^2$; In fact, $|H|^2$ coincides with the \textit{Graham-Lee transverse curvature} of the defining function; see \cite{son2019semi}.
As a consequence of Theorem~\ref{thm:m1} and Proposition~\ref{prop:cconvex}, we have the following
\begin{corollary}\label{cor13}
	Let $(M,\theta)$ be given as above and assume that \eqref{e:10} holds. If $M$ is $\CC$-convex then $(M,\theta)$ has $\frac12$-positive holomorphic sectional curvature. More precisely, for any $(1,0)$-vector $Z \in T^{(1,0)}(M)$, $Z\ne 0$,
	\begin{align}
		K(Z) \geqslant \frac{1}{2}|\partial \vr|^{-2} \geqslant \frac{1}{2}|A(Z)|.
	\end{align}
	The inequalities are sharp.
\end{corollary}

Corollary~\ref{cor13} generalizes \cite[Theorem~1.2]{chanillo2012embedded} and \cite{li2011cr} about the positivity of the Tanaka-Webster scalar curvature on real ellipsoids. Recall that an ellipsoid $E$ in $\CC^2 \cong \RR^4$ is a compact real hypersurface given by the equation
\begin{align*}
	ax^2 + by^2 + cu^2 + d v^2 = 1.
\end{align*}
In complex coordinates $z = x + iy, w = u+iv$, we have
\begin{align*}
	E = \left\{(z,w) \in \CC^2 \colon \vr = \alpha |z|^2 + \beta |w|^2 + \Re (\gamma z^2 + \sigma w^2) -1 = 0\right\},
\end{align*}
where $\alpha = (a+b)/2, \beta = (c+d)/2, \gamma = (a-b)/2$, and $\sigma = (c-d)/2$. Put $\theta = \iota^{\ast}(i\bar{\partial}\vr)$, where $\iota \colon E \to \CC^2$ is the inclusion. Then $(E,\theta)$ is semi-isometrically immersed in $\CC^2$ with the metric $\alpha dz \otimes d\zba + \beta dw \otimes d\wba$. As $E$ is strictly convex, it is strictly $\CC$-convex in the sense of \cite{andersson2012complex} (hence pseudohermitian $\CC$-convex in the sense of Definition~\ref{def1} by Proposition~\ref{prop:2c}). Thus
\begin{align}
	R = 2K
	\geqslant 
	\frac{1}{|\partial\vr|^2} = \frac{\alpha\beta}{\beta|\vr_z|^2 + \alpha|\vr_w|^2}.
\end{align}
In particular, $R>0$, which was proved before in \cite{li2011cr} and \cite{chanillo2012embedded}.

Evidently, a real ellipsoid in general dimension is stricly convex and the same argument as above also shows that its holomorphic sectional curvature is bounded below by $\frac12 |\partial\vr|^{-2} > 0$.

A motivation for studying lower estimates of the scalar curvature is to bound the first positive eigenvalue $\lambda_1(\Box_b)$ of the Kohn-Laplacian $\Box_b : = \bar\partial^{\ast} \bar\partial + \bar\partial\bar\partial^{\ast}$ from below. Such a lower bound was first obtained in \cite{chanillo2012embeddability} under the nonnegativity of the CR Paneitz operator (see \cite{chanillo2012embeddability} and \cite{takeuchi} for more details on the CR Paneitz operator and \cite{lsw} for higher dimensional version of this estimate). Combining Corollary~\ref{cor13}, Chanillo, Chiu, and Yang's lower bound \cite{chanillo2012embeddability}, and the recent result of Takeuchi \cite{takeuchi}, we obtain the following 
\begin{corollary}\label{cor15}
	Let $M$ be a compact strictly pseudoconvex pseudohermitian real hypersurface $\CC^2$ defined by $\vr = 0$ and let $\theta = i\bar\partial\vr$. Assume that \eqref{e:10} holds. If $M$ is $\CC$-convex, then the first positive eigenvalue $\lambda_1(\Box_b)$ of the Kohn-Laplacian acting on functions satisfies
	\begin{align}\label{e:12}
		\lambda_1(\Box_b) \geqslant \frac{1}{2} \min_{M} \left(\frac{1}{ |\partial\vr|^2}\right).
	\end{align}
\end{corollary}
\begin{remark}
	When $\vr_{j\kba} = \delta_{jk}$, the Kroneker symbols, it was proved by Li, Lin, and the author in \cite{li2018} that 
	\begin{align}\label{e:18}
		\lambda_1(\Box_b) \leqslant \mathrm{average}_{M} \left(\frac{1}{ |\partial\vr|^2}\right)
	\end{align}
	where the averaging is taken with respect to the volume form $\theta \wedge d\theta$. In \eqref{e:12}, due to the factor $\frac12$ on the right-hand side, the estimate is not expected to be sharp. In view of \eqref{e:18}, it is interesting to know if one can improve \eqref{e:12} with the constant 1 on the right-hand side?
\end{remark}

\section{Backgrounds and basic results}
\subsection{Pseudohermitian structure on real hypersurfaces}
In this section, we introduce several basic notions and results of pseudohermitian geometry of nondegenerate real hypersurfaces in a complex manifold. For further details, we refer the readers to \cite{tanaka1975differential} and \cite{webster1978pseudo}.

Let $M\subset \mathcal{X}$ be a strictly pseudoconvex real hypersurface in a complex manifold. The CR structure on $M$ is induced from the complex structure on $\mathcal{X}$, namely,
\begin{align*}
	T^{(1,0)} (M) : = T^{(1,0)}(\mathcal{X}) \cap \CC T(M).
\end{align*}
The real hyperplane $H(M):=\Re \left(T^{(1,0)} (M)\right) \subset T(M)$ is the kernel of a nonvanishing real 1-form~$\theta$. By strict pseudoconvexity, $d\theta$ is positively or negatively definite when restricted to $H(M)$. For simplicity, we assume that $d\theta$ is positive definite on $H(M)$. The triple $(M,T^{(1,0)} (M), \theta)$ is called a pseudohermitian manifold \cite{webster1978pseudo}.

Fix a pseudohermitian structure $\theta$, there exists a unique real vector field $T$ transverse to $H$, the Reeb vector field, such that
\begin{align*}
	T \rfloor d\theta =0,\quad 
	\theta(T) = 1.
\end{align*}
Let $Z_{\alpha}$, $\alpha =1,2,\dots, n$, be a local frame for $T^{(1,0)}(M)$ and let $Z_{\aba}=\overline{Z_{\alpha}}$ be the conjugates. Then $\{Z_{\alpha},Z_{\aba},T \}$ is a local frame for $\CC T(M)$. Its dual frame $\{\theta^{\alpha},\theta^{\aba},\theta\}$ is called an admissible coframe if
\begin{align*}
	d\theta = i h_{\alpha\bba} \theta^{\alpha} \wedge \theta^{\bba},
\end{align*}
where $h_{\alpha\bba}$ is a positive definite Hermitian matrix. We use $h_{\alpha\bba}$ and its inverse transpose $h^{\alpha\bba}$ to lower and raise the Greek indices.

On a pseudohermitian manifold, there is a canonical linear connection---the Tanaka-Webster connection \cite{tanaka1975differential,webster1978pseudo}---which can be defined by
\begin{align*}
	\nabla Z_{\alpha} = \omega_{\alpha}{}^{\beta} Z_{\beta},
	\quad
	\nabla Z_{\aba} = \omega_{\aba}{}^{\bba} Z_{\bba},
	\quad
	\nabla T = 0, 
\end{align*}
where the connection forms $\omega_{\alpha}{}^{\beta}$'s satisfy the following structural equations
\begin{align*}
	d\theta ^{\beta} & = \theta^{\alpha} \wedge \omega_{\alpha}{}^{\beta} + A^{\beta}{}_{\aba} \theta \wedge \theta^{\aba}, \\
	dh_{\alpha\bba} & = \omega_{\alpha}{}^{\gamma} h_{\gamma \bba} + h_{\alpha\gba} \omega _{\bba}{}^{\gba}.
\end{align*}
The tensor $A_{\alpha\beta}$ is the pseudohermitian torsion, which is shown to be symmetric: $A_{\alpha\beta} = A_{\beta\alpha}$ \cite{webster1978pseudo}.
The curvature form
$\Omega{_{\alpha}^{\beta}} := d\omega_{\alpha}{}^{\beta}
- \omega_{\alpha}{}^{\gamma} \wedge \omega_{\gamma}{}^{\beta}$
of the Tanaka-Webster connection satisfies
\begin{align*}
	\Omega_{\alpha}{}^{\beta}
	&= R_{\alpha}^{\beta}{}_{\gamma\sba } \theta{^{\gamma}} \wedge \theta{^{\sba}}
	- \nabla^{\beta} A_{\alpha\gamma} \theta \wedge \theta^\gamma
	+ \nabla_{\alpha} A^{\beta}{}_{\gba} \theta \wedge \theta^{\gba} \notag \\
	&\qquad - iA_{\alpha\gamma} \theta^{\gamma} \wedge \theta^{\beta}
	+ i h_{\alpha\gba} A^{\beta}{}_{\bar{\vr }}
	\theta^{\gba} \wedge \theta^{\bar{\vr }}.
\end{align*}
The tensor $R_{\alpha}{}^{\beta}{}_{\gamma\sba}$ is called the \emph{Tanaka-Webster curvature}. It satisfies the following symmetries:
\begin{equation*}
	R_{\alpha\bba\gamma\sba}
	=R_{\gamma\bba\alpha\sba}
	=R_{\alpha\sba\gamma\bba}.
\end{equation*}
As usual, contraction of indices gives the Ricci tensor $R_{\alpha\bba} = h^{\gamma\sba}R_{\alpha\bba\gamma\sba}$ and the scalar curvature $R = h^{\alpha\bba} R_{\alpha\bba}$. The Tanaka-Webster holomorphic sectional curvature is defined similarly to the Hermitian or K\"ahler case by equation \eqref{e:secdef} in the introduction. For CR 3-manifolds, these 3 curvature quantities are essentially the same.

\edit{The fundamental Chern-Moser invariant of Levi-nondegenerate CR manifolds can be ``represented'' by the completely tracefree part of the Tanaka-Webster curvature, the \textit{Chern-Moser-Weyl tensor} $S_{\beta}{}^{\alpha}{}_{\gamma\sba}$ \cite{webster1978pseudo}. Precisely,
	\begin{align*}
		S_{\beta}{}^{\alpha}{}_{\gamma\bar{\sigma}}
		&=
		R_{\beta}{}^{\alpha}{}_{\gamma\bar{\sigma}}
		- \frac{R_{\beta}{}^{\alpha} h_{\gamma\bar{\sigma}} +R_{\gamma}{}^{\alpha} h_{\beta\bar{\sigma}} + \delta_{\beta}^{\alpha}R_{\gamma\bar{\sigma}} + \delta_{\gamma}^{\alpha} R_{\beta\bar{\sigma}}}{n+2}
		+ \frac{R(\delta_{\beta}^{\alpha}h_{\gamma\bar{\sigma}}+\delta_{\gamma}^{\alpha}h_{\beta\bar{\sigma}})}{(n+1)(n+2)},
	\end{align*}
	where $\dim M = 2n+1\geqslant 5$.}

Let $f$ be a smooth function on a compact pseudohermitian manifold $M$. We can write
\begin{equation*}
	df = f_{\alpha}\theta^{\alpha} + f_{\bba}\theta^{\bba} + f_0 \theta.
\end{equation*}
The \textit{Kohn-Laplacian} acting on funtions $\Box_b \colon C^{\infty}(M,\CC) \to C^{\infty}(M,\CC)$, $\Box_b f = \bar\partial_b^{\ast} \bar\partial_b f$, can be expressed locally by
\begin{equation*}
	\Box_b f = - f_{\aba,}{}^{\aba},
\end{equation*}
where an index preceded by a coma indicates a Tanaka-Webster covariant derivative. For compact embeddable strictly pseudoconvex CR manifolds, the spectral theory of $\Box_b$ is well-understood. Combining Chanillo, Chiu, and Yang \cite{chanillo2012embeddability} and Takeuchi \cite{takeuchi}, we find that the first positive eigenvalue $\lambda_1(\Box_b)$ on embeddable CR 3-manifold satisfies
\begin{equation}\label{ed}
	\lambda_1(\Box_b) \geqslant \frac{1}{2}\min_M R.
\end{equation}
This is meaningful when $(M,\theta)$ has positive Tanaka-Webster scalar curvature. For the case $n := \dim _{C\!R} M \geqslant 2$, if the Tanaka-Webster Ricci curvature is bounded below by $\kappa > 0$, then 
\begin{equation} 
	\lambda_1(\Box_b) \geqslant \frac{n}{n+1} \kappa.
\end{equation} 
See \cite{lsw}.
\subsection{Semi-isometric CR immersions and the Gauss equations}
Let $M \subset \CC^{n+1}$ be a strictly pseudoconvex CR manifold and let $\vr$ be a defining function for $M$, i.e., $M = \{\vr = 0\}$ and $d\vr \ne 0 $ along $M$. If in addition $\vr$ is strictly plurisubharmonic, then it induces a K\"ahler metric $\omega = i\partial\bar\partial\vr$ on a neighborhood $U$ of $M$. The pseudohermitian structure $\theta: = \iota ^{\ast}(i\bar\partial\vr)$ and the K\"ahler form $\omega$ satisfy $d\theta = \iota^{\ast} \omega$ and thus $\iota \colon M \to U$ is a semi-isometric immersion in the terminology of \cite{son2019semi}. In fact, we have the following
\begin{definition}[\cite{son2019semi}] Let $(M,\theta)$ be a strictly pseudoconvex pseudohermitian manifold, $(\mathcal{X},\omega)$ a K\"ahler manifold, and $F\colon M\to (\mathcal{X},\omega)$ a smooth CR mapping. We say that $F$ is \textit{semi-isometric} if
	\begin{equation}\label{e:semi-isometric}
		d\theta = F^{\ast} \omega.
	\end{equation}
\end{definition}
If $M\subset \mathcal{X}$ and if the inclusion $\iota \colon (M,\theta) \to (\mathcal{X},\omega)$ is semi-isometric, then we say that $(M,\theta)$ is a pseudohermitian submanifold of $(\mathcal{X},\omega)$. In this situation, the pseudohermitian geometry of $(M,\theta)$ and the K\"ahler geometry of the ambient manifold $(\mathcal{X},\omega)$ are related via the Gauss equations \cite{son2019semi}. We describe this in more detail as follows. Let $\nabla$ and $\widetilde{\nabla}$ be the Tanaka--Webster connection of $(M , \theta)$ and the Chern connection of $(\mathcal{X},\omega)$, respectively. Then the second fundamental form of $M$ is defined by the Gauss formula (cf. \cite{son2019semi})
\begin{equation*} 
	\II(Z,W) := \widetilde{\nabla}_{\widetilde{Z}}\widetilde{W} - \nabla_ZW.
\end{equation*} 
Here $\widetilde{Z}$ and $\widetilde{W}$ are smooth extensions of $Z$ and $W$ to a neighborhood of a point of $M$ in $\mathcal{X}$.
Taking the trace of $\II$ on {horizontal directions}, we obtain the $(1,0)$-mean curvature vector field~
$H$. Namely,
\begin{equation}
	H: = \frac1n\sum_{\alpha =1}^{n} \II(Z_{\aba} , Z_{\alpha}).
\end{equation}

Basic properties of $\II$ and $H$ were given in \cite{son2019semi}. For example, $\II$ is nonsymmetric, but for $ Z,W \in T^{(1,0)}(M)$, it holds that
\begin{align}
	\II(Z,W) = \II(W,Z),
\end{align}
while for $Z \in T^{(1,0)}(M)$ and $\Wba \in T^{(0,1)}(M)$,
\begin{align}
	\II(Z, \Wba) - \II(\Wba,Z)
	=
	- i \langle Z, \Wba \rangle T,
\end{align}
where $T$ is the Reeb vector field. Moreover,
\begin{align}
	\II(Z,\Wba) 
	=
	\langle Z ,\Wba \rangle \overline{H}.
\end{align}
We say that $\iota$ is \textit{totally umbilical} if $\II(Z,W) =0 $ for all $(1,0)$-vectors $Z$ and $W$.

We shall use the following conventions for the curvature operator and torsion of $\nabla$:
\begin{align*}
	R(X,Y)Z & = \nabla_X\nabla_YZ - \nabla_Y \nabla_X Z - \nabla_{[X,Y]}Z,\\
	\mathbb{T}_{\nabla} (X,Y) & = \nabla_XY - \nabla_{Y}X - [X,Y].
\end{align*}
Denote by $\tau$ the pseudohermitian torsion, i.e.,
\begin{equation}
	\tau X := \mathbb{T}_{\nabla} (T,X).
\end{equation}
The Levi metric on $M$ and the K\"ahler metric on $\mathcal{X}$ will be denoted by $\langle \cdot , \cdot \rangle$. The curvature of the Chern connection of $\omega$ will be denoted by $\widetilde{R}$. The Gauss equations are as follows:
\begin{proposition}[Gauss equations \cite{son2019semi}]\label{prop:ge}
	Let $\iota \colon (M, \theta) \hookrightarrow (\mathcal{X},\omega)$ be a pseudohermitian CR submanifold of a Kähler manifold. Let $R$ and $\widetilde{R}$ be the curvature operators of the Tanaka--Webster and Chern connection on $M$ and $\mathcal{X}$, respectively. Then
	\begin{enumerate}
		\item for $X,Z \in \Gamma(T^{(1,0)}(M))$ and $\overline{Y},\overline{W} \in \Gamma(T^{0,1} (M)) $, the following Gauss equation holds:
		\begin{align}\label{e:gauss}
			\langle \widetilde{R}(X,\overline{Y}) Z, \Wba\rangle
			& =
			\langle R(X,\overline{Y}) Z, \Wba\rangle
			+
			\langle \II (X,Z) , \II (\overline{Y}, \Wba) \rangle \notag \\
			& \qquad - |H |^2 \left(\langle \overline{Y} , Z \rangle \langle X ,\Wba \rangle + \langle X , \overline{Y} \rangle \langle Z , \Wba \rangle \right),
		\end{align}
		\item for $X,Z \in \Gamma(T^{(1,0)}(M))$,
		\begin{equation}\label{e:gausstorsion}
			\langle \tau X , Z \rangle 
			=
			-i \langle \II(X,Z) , \overline{H} \rangle.
		\end{equation}
	\end{enumerate}
\end{proposition}
The Gauss equation \eqref{e:gauss} is similar to several versions in the literature (cf. \cite{ehz}) and has found several applications \cite{son2019semi,reiter2019chern}. For example, Reiter and the author \cite{reiter2019chern} used this equation to establish an explicit and concise formula for the well-known Chern-Moser-Weyl tensor. 
On the other hand, equation \eqref{e:gausstorsion}, which has no Riemannian counterpart, relates the (intrinsic) torsion of the submanifold and the (extrinsic) second fundamental form of the immersion. This equation will also be of importance for us.

As a consequence of the Gauss equations we have the following
\begin{proposition}\label{prop:22} Let $\iota \colon (M, \theta) \hookrightarrow (\mathcal{X},\omega)$ be a real hypersurface of a Kähler manifold with $d\theta = \iota^{\ast} \omega$ and let $H$ be the $(1,0)$-mean curvature vector of $M$. Let $K$ and $\tilde{K}$ be the holomorphic sectional curvatures of $\theta$ and $\omega$, respectively, and let $A$ be the holomorphic sectional torsion of~$\theta$. Then
	\begin{align}\label{e:36}
		K(Z) + \frac{1}{2} |A(Z)|^2/ |H|^2 = \frac{1}{2
		}\tilde{K}(Z) + |H|^2 \quad \forall \, Z \in T^{(1,0)}(M).
	\end{align}
	Moreover, $(M,\theta)$ has vanishing pseudohermitian torsion if and only if $\iota \colon M \to \mathcal{X}$ is totally umbilical.
\end{proposition}
\begin{proof}
	Let $Z \in T^{(1,0)}(M)$ with $|Z| = 1$. From Gauss equation \eqref{e:gauss} we get
	\begin{align}\label{e:38}
		\frac{1}{2}\tilde{K}(Z)
		=
		K(Z)
		+
		\frac{1}{2}|\II(Z,Z)|^2 - |H|^2.
	\end{align}
	On the other hand, from \eqref{e:gausstorsion} and \eqref{e5}, we also have that
	\begin{align*}
		i A_{\alpha\beta} \zeta^{\alpha}\zeta^{\beta}
		=
		i\langle \tau Z, Z \rangle 
		=
		\langle \II(Z,Z) , \overline{H} \rangle 
		=
		\sqrt{2}\, h(Z,Z) |H|.
	\end{align*}
	But $\II(Z,Z)$ belongs to $N^{(1,0)}(M)$ which has 1-dimensional fibers, we have
	\begin{align*}
		\II(Z,Z) = \sqrt{2}\, h(Z,Z) H/|H|,
	\end{align*}
	and therefore, since $|Z| = 1$, we have
	\begin{align*}
		|\II(Z,Z)|^2
		=
		2\,|h(Z,Z)|^2
		=
		|A(Z)|^2\bigl/ |H|^2.
	\end{align*}
	Plugging this into \eqref{e:38} we obtain \eqref{e:36}. Moreover, $\theta$ has vanishing torsion iff $A(Z) = 0$ iff $\II(Z,Z) = 0$ for all $Z \in T^{(1,0)}(M)$. But this is also equivalent to $\II(Z,W) = 0$ for all $Z,W\in T^{(1,0)}(M)$ since $\II$ is symmetric and bilinear when restricted to $T^{(1,0)}(M)$, i.e., $M$ is totally umbilical. The proof is complete.
\end{proof}

\subsection{Pseudohermitian $\CC$-convexity}\label{subsec:convex}
A domain $\Omega \subset \mathbb{CP}^{n+1}$ is $\CC$-convex if any non-empty intersection of $\Omega$ with a complex line is connected and simply connected. This notion of convexity plays an important role in several problems in complex analysis related to
the geometry of (weakly or strongly) pseudoconvex domains in complex euclidean or projective spaces; see \cite{andersson2012complex} for some motivations and more details. For bounded domains with $C^2$-smooth boundaries in $\CC^{n}\subset \mathbb{CP}^n$, there is also a differential condition for the $\CC$-convexity which involves the Hessian of a defining function. This condition is described as follows. If $\vr$ is a real-valued $C^2$-smooth function, then the Hessian of $\vr$ is the following quadratic form
\begin{align}
	\mathrm{Hess}_{\vr}(a; \eta) := \sum_{j,k}\Re \left( \vr_{jk}(a) \eta_j \eta_k\right) + \sum_{j,k} \vr_{j\kba}(a) \eta_j \eta_{\kba}, 
	\quad \eta \in \CC^n.
\end{align}
The boundary characterization of the $\CC$-convexity is as follows.
\begin{proposition}[see Theorem~2.5.18 of \cite{andersson2012complex}]\label{prop:cconvex}
	Suppose that $\Omega$ is $C^2$-smoothly bounded domain with a real-valued $C^2$-smooth defining function $\vr$, i.e., $\Omega = \{\vr < 0\}$ with $d\vr$ does not vanish on the boundary of $\Omega$. Then the following are equivalent.
	\begin{enumerate}
		\item The domain $\Omega$ is $\CC$-convex.
		\item The restriction of the Hessian of $\vr$ at any boundary point $p\in \partial\Omega $ to the complex tangent plane through $p$ is positive semi-definite.
	\end{enumerate}
\end{proposition}
The condition that $\mathrm{Hess}_{\vr}$ is positive (semi-)definite on the complex tangent plane is offen called the ``Behnke-Peschl'' condition which implies the weak pseudoconvexity of the domain. Clearly, the $\CC$-convexity is \textit{not} invariant under general biholomorphic mappings of $\CC^n$ or $\mathbb{CP}^n$. Thus it cannot be generalized for real hypersurfaces in an arbitrary complex manifold without additional structures (e.g. a K\"ahler metric). For subsets of complex projective space, this convexity gives a rather strict constrain. For example, a $\CC$-convex domain in $\CC^{N}$ must be diffeomorphic to the ball (see \cite{andersson2012complex}).

The notion of pseudohermitian $\CC$-convexity in Definition~\ref{def1} is analogous to the notion of $\CC$-convexity in Proposition~\ref{prop:cconvex}, but the former depends essentially on the pseudohermtian structure on $M$ and the K\"ahler metric on the ambient space. We prove below that for a class of pseudohermitian real hypersurfaces in $\CC^{n+1}$ these two notions of convexity coincide.
\begin{proposition}\label{prop:2c}
	Let $M\subset \mathbb{C}^{n+1}$ be a strictly pseudoconvex real hypersurface defined by $\vr = 0$ with $d\vr \ne 0$. Suppose that $\vr_{j\kba} = a_{j\kba} + o(\vr)$ along $M$, where $[a_{j\kba}]$ is a hermitian matrix of constants. Let $\theta = \iota^{\ast}(i\bar{\partial}\vr)$. Then $(M,\theta)$ is a semi-isometric real hypersurface in $\mathbb{C}^{n+1}$ equipped with the hermitian metric $a_{j\kba} dz^j \otimes dz^{\kba}$. Moreover, the following are equivalent.
	\begin{enumerate}
		\item $M$ is $\CC$-convex.
		\item $(M,\theta)$ is pseudohermitian $\CC$-convex in $\left(\mathbb{C}^{n+1}, a_{j\kba} dz^j \otimes dz^{\kba}\right)$ in the sense of Definition~\ref{def1}.
	\end{enumerate}
\end{proposition}
\begin{proof}
	The proof follows from explicit formulas for the second fundamental form in \cite{son2019semi} and \cite{reiter2019chern}.
	Notice that the Levi-matrix $h_{\alpha \bba}$ in the frame $Z_{\alpha}: = \vr_{w} \partial_{\alpha} - \vr_{\alpha} \partial_w$ is given by
	\begin{equation*}
		h_{\alpha\bba}: = -id\theta(Z_{\alpha},Z_{\bba}) = \vr_{Z\bar{Z}} (Z_{\alpha},Z_{\bba}),
	\end{equation*} 
	where $\vr_{Z\bar{Z}}$ is the Hermitian part of the Hessian of $\vr$. Following \cite{reiter2019chern}, we define the second order differential operator
	\begin{equation*}
		D_{\alpha\beta}^{\varrho}
		:=
		\vr_{w}^2\partial_{\alpha}\partial_{\beta} - \vr_{w}\varrho_{\alpha}\partial_{w}\partial_{\beta}
		- \vr_{w}\varrho_{\beta}\partial_{w}\partial_{\alpha} + \varrho_{\alpha}\varrho_{\beta}\partial_w^2,
	\end{equation*}
	which satisfies
	\begin{equation*} 
		D_{\alpha\beta}^{\vr}(\varphi) = \varphi_{ZZ}(Z_{\alpha}, Z_{\beta}).
	\end{equation*} 
	From \cite[Prop. 2.2]{reiter2019chern}, we have
	\begin{equation*}
		\II(Z_{\alpha} , Z_{\beta})
		=
		- \left(\varrho^{\kba} D^{\varrho}_{\alpha\beta}(\varrho_{\kba}) - h_{\alpha\beta}\right) H,
	\end{equation*}
	On the other hand, from \cite{son2019semi} we have
	\begin{align*}
		\II(Z_{\alpha},Z_{\bba})
		=
		-\vr_{Z\Zba} (Z_{\alpha},Z_{\bba})\, \overline{H}.
	\end{align*}
	Thus, for $Z = \zeta^{\alpha} Z_{\alpha} \in T_p^{(1,0)}(M)$, we have from \eqref{e5}
	\begin{multline*}
		h(Z+\Zba, Z+ \Zba)\bigl|_p 
		=\\
		\sqrt{2}\,|H| \left(\Re \left(-\vr^{\kba} D_{\alpha\beta}(\vr_{\kba})\zeta^{\alpha}\zeta^{\beta} + h_{\alpha\beta}\zeta^{\alpha}\zeta^{\beta}\right) + h_{\alpha\bba} \zeta^{\alpha}\zeta^{\bba}\right)\biggl|_p.
	\end{multline*}
	If $\vr_{j\kba}$'s are constants, then $D_{\alpha\beta}(\vr_{\kba}) = 0$ and thus, the scalar-valued second fundamental form $h$ restricted to $H(M)$ reduces essentially to the restriction of the usual Hessian of $\vr$. Precisely,
	\begin{align*}
		h(Z+\Zba, Z+ \Zba)\bigl|_p 
		& = 
		\sqrt{2}\,|H| \left(h_{\alpha\bba}\zeta^{\alpha}\zeta^{\bba} + \Re \left( h_{\alpha\beta}\zeta^{\alpha}\zeta^{\beta}\right) \right)\biggl|_p \notag \\
		& = 
		\sqrt{2}\,|H|\, \mathrm{Hess}_{\vr}(p; Z).
	\end{align*}
	Hence, when $\vr_{j\kba}$'s are constants, the pseudohermitian $\CC$-convexity condition in Definition~\ref{def1} is exactly as the condition for domains in $\CC^{n+1}$ in Proposition~\ref{prop:cconvex}. The proof is complete.
\end{proof}
\begin{proposition}\label{prop:26}
	Let $(M,\theta)$ is be a semi-isometrically immersed real hypersurface in $(\mathcal{X}, \omega)$ with $(1,0)$-mean curvature vector $H$. Then $(M,\theta)$ is pseudohermitian $\CC$-convex if and only if 
	\begin{align}\label{ea}
		|A(Z)| \leqslant |H|^2 \quad \forall \, Z \in T^{(1,0)}(M),\, Z \ne 0.
	\end{align}
	Moreover, $(M,\theta)$ is stricly pseudohermitian $\CC$-convex iff the strict inequality in \eqref{ea} holds.
\end{proposition}
Thus, if $(M,\theta)$ has vanishing pseudohermitian torsion, then for any (real codimension one) semi-isometric CR immersion $\iota \colon (M,\theta) \to (\mathcal{X},\omega)$ the image $\iota(M)$ must be pseudohermitian $\CC$-convex in $(\mathcal{X},\omega)$.
\begin{proof}
	Assume that $M$ is pseudohermitian $\CC$-convex. By the Gauss equation, we have for $Z \in T^{(1,0)}(M)$ with $|Z| = 1$,
	\begin{align*}
		A(Z)
		=
		i\langle \tau Z, Z \rangle 
		=
		\langle \II(Z,Z), \overline{H} \rangle.
	\end{align*}
	By the pseudohermitian $\CC$-convexity, the real-valued form $h$ defined in \eqref{e5} is positive semi-definite on $H(M)$. Applying this positivity to the vector $iZ - i \overline{Z}$, we have
	\[
	h(iZ - i \overline{Z}, iZ - i \overline{Z}) \geqslant 0
	\]
	which is equivalent to
	\[
	\Re\, \langle \II(Z,\Zba), H \rangle -\Re \, \langle \II(Z,Z), \overline{H} \rangle 
	\geqslant 0.
	\]
	Thus,
	\begin{equation*}
		\Re A(Z) 
		= \Re \, \langle \II(Z,Z), \overline{H} \rangle 
		\leqslant
		\Re\, \langle \II(Z,\Zba), H \rangle 
		=
		|H|^2.
	\end{equation*}
	Furthermore, by a difference in homogeneity, this is equivalent to
	\begin{align*}
		|A(Z)| \leqslant |H|^2 \quad \forall \ Z \in T^{(1,0)}(M),
	\end{align*}
	\edit{as desired.}	The proof of the converse uses similar calculation. We leave the details to the readers.
\end{proof}
\begin{example}[cf. Example~2.5.14 in \cite{andersson2012complex}]\rm Consider for $t\in [0,1]$ the Hargtog domain
	\begin{align*}
		E_t = \left\{(z,w) \in \CC^2 \colon -1 + |z|^2 + |w|^2 + t\left(\Re z^2\right)^2 <0\right\}.
	\end{align*}
	Evidently, $E_t$ is bounded. It has been shown in \cite{andersson2012complex} that $E_t$ has smooth boundary and is convex for $0< t < 3/4$ and non-convex for $3/4\leqslant t \leqslant 1$. However, it is $\CC$-convex for $t\in [0,1]$ and thus the $\CC$-convexity is a weaker condition than the ``usual'' convexity.
	
	For any $t$, the Tanaka-Webster scalar curvature $R$ of $\theta:=i\bar{\partial}\vr$ at a point $(0,e^{i\tau})$ is $2(1-t)$. If $t = 1$ then $R=0$ on the circle $(0,e^{i\tau})$. Thus, the $\CC$-convexity alone is not enough to have $R$ be bounded below by the Graham-Lee transverse curvature. We point out that each $E_t$ can be semi-isometrically embedded as real codimension 3 CR submanifold of $\CC^3$ with the flat metric by the embedding $(z,w) \mapsto (z,w, \sqrt{t/2}\,z^2)$. But the notion of pseudohermitian $\CC$-convexity is not well-defined for higher codimensional immersions.
\end{example}
\section{Proofs of Theorem~\ref{thm:m1} and its corollaries}

\begin{proof}[Proof of Theorem~\ref{thm:m1}] The first statement follows from Proposition~\ref{prop:26}. In particular, if $(M,\theta)$ is pseudohermitian $\CC$-convex, then $|A(Z)|\leqslant |H|^2$. Plugging this into \eqref{e:36} we obtain
\begin{align*}
	K(Z) \geqslant \frac12 \tilde{K}(Z) + \frac{1}{2} |H|^2
\end{align*}
which proves inequality \eqref{e:mest}.

To show that inequalities \eqref{e:mest} and \eqref{e:mesta} are sharp, we consider the following real hypersurface which has been studied in various papers, e.g., \cite{reiter2019chern} and the references therein. Precisely, let
\begin{equation}\label{e:ex}
	\vr (z) : = \sum_{j=1}^{n+1}|z_j|^2 + \Re \sum_{j=1}^{n+1} z_j^2 -1,
\end{equation}
and let $E =\left\{z\in \CC^{n+1} \colon \vr(z) = 0\right\}$ and $\theta_E: = \iota^{\ast}(i\bar{\partial}\vr)$ be a pseudohermitian structure on $E$. In the frame $Z_{\alpha}:=\vr_{w} \partial_{\alpha} - \vr_{\alpha}\partial_w$, we have
\begin{equation*}
	|\partial \vr|^2 = \sum_{j=1}^{n+1} |z_j + \zba_j|^2 = 2\vr +2, \quad 	h_{\alpha\beta} = h_{\alpha\bba} = \vr_w^2\delta_{\alpha\beta} + \vr_{\alpha}\vr_{\beta},
\end{equation*}
where $\alpha, \beta =1,2,\dots, n$ and $w: = z_{n+1}$. The Tanaka-Webster pseudohermitian curvature tensor in this frame is computed in, e.g., \cite{reiter2019chern}. Precisely,
\begin{equation*}
	R_{\alpha\bba\gamma\sba}
	=
	- \frac{h_{\alpha\gamma}h_{\bba \sba}}{2} + \frac{h_{\alpha\bba}h_{\gamma\sba} + h_{\alpha\sba}h_{\gamma\bba}}{2}.
\end{equation*}
Therefore, for $Z = \zeta^{\alpha} Z_{\alpha} \in T^{(1,0)}(M)$ with $|Z|^2 = h_{\alpha\bba} \zeta^{\alpha} \zeta^{\bba} = 1$, we have
\begin{align}\label{e:92}
	K(Z) = \frac{1}{2} R_{\alpha\bba\gamma\sba}\zeta^{\alpha} \zeta^{\bba} \zeta^{\gamma} \zeta^{\sba} = \frac12 - \frac{1}{4}\left|h_{\alpha\gamma} \zeta^{\alpha} \zeta^{\gamma}\right|^2.
\end{align}

Observe that $(E,\theta_E)$ is semi-isometrically embedded in $\CC^{n+1}$ with the euclidean metric (\edit{thus}, $\tilde{K} = 0$) and 
\begin{align}\label{ec}
	|H|^2 = |\partial \vr|^{-2} = \frac{1}{2}.
\end{align}
To show that $E$ is pseudohermitian $\CC$-convex, we observe that
\begin{align*}
	A_{\alpha\beta} = \frac12 h_{\alpha\beta},
\end{align*}
and hence
\begin{align*}
	A(Z) = \frac{1}{2}h_{\alpha\beta} \zeta^{\alpha}\zeta^{\beta}
	=
	\frac{1}{2} \vr_{w}^2 \sum_{\alpha=1}^n (\zeta^{\alpha})^2 + \frac{1}{2} \left(\sum_{\alpha = 1}^n \vr_{\alpha} \zeta ^{\alpha}\right)^2,
\end{align*}
Thus,
\begin{align}\label{e:95}
	|A(Z)|
	\leqslant
	\frac{1}{2} |\vr_w|^2 \sum_{\alpha=1}^n |\zeta^{\alpha}|^2 + 	\frac{1}{2} \left|\sum_{\alpha = 1}^n \vr_{\alpha} \zeta^{\alpha}\right|^2
	=
	\frac{1}{2}h_{\alpha\bba}\zeta^{\alpha}\zeta^{\bba} = \frac{1}{2},
\end{align}
with the equality occurs when, e.g., $\zeta^{\alpha}$'s are all real, as $\vr_{\alpha}$'s are real.
This and \eqref{ec} mean that $E$ is (non-strict) pseudohermitian $\CC$-convex by Proposition~\ref{prop:26}, as desired.
Equations \eqref{e:92} and \eqref{e:95} also show that 
\begin{equation}
	\frac{1}{4}\leqslant K(Z) \leqslant \frac{1}{2}.
\end{equation}
Thus, for $(E,\theta_E)$ the equalities in \eqref{e:mesta} and \eqref{e:mest} do occur.
\end{proof}
\begin{remark}
	It is easy to see that as a consequence of the Cauchy-Schwarz inequality the CR holomorphic bisectional of $(E,\theta_E)$ in the proof above is nonnegative but not strict positive.
\end{remark}
As already briefly discussed in the introduction, on a CR manifold the positivity of the Tanaka-Webster scalar curvature does not imply the existence of positive scalar curvature Riemannian metric. In particular, it does not implies the positivity of the scalar curvature of any adapted Webster metric. In \cite{cao2020}, Cao, Chang, and Chen introduced the following notion exhibiting the importance of the torsion as follows: A closed CR 3-manifold is said to have $C_0$-positive pseudohermitian curvature if there exists a pseudohermitian structure $\theta$ having curvature and torsion satisfy
\begin{align}\label{e:c0pos}
	R|Z|^2 + C_0 \mathrm{Tor}(Z,Z) > 0 \quad \forall Z \in T^{(1,0)}(M).
\end{align}
They proved \cite[Theorem~1.1]{cao2020} that if a closed CR 3-manifold is $C_0$ positive for $C_0 \geqslant \frac12$, then the adapted Riemannian metric
\begin{align}
	g_{\lambda}:=d\theta + \lambda^{-2} \theta \otimes \theta 
\end{align}
has positive scalar curvature \edit{for some constant $\lambda$}.

Cao et al's notion of $C_0$-positivity can be generalized to higher dimensions as follows.
\begin{definition}
	A pseudohermitian manifold $(M^{2n+1},\theta)$ of real dimension $2n+1\geqslant 3$ is said to have \emph{$C_0$-positive holomorphic sectional curvature} if $\forall Z \in T^{(1,0)}(M),\, Z\ne 0$,
	\begin{equation}\label{e:c0pos1}
		K(Z) - C_0 |A(Z)| > 0 .
	\end{equation}
\end{definition}
When $n=1$, \eqref{e:c0pos} is equivalent to \eqref{e:c0pos1} since $R|Z|^2 = 2K(Z)$ and $\tor(Z,Z) = 2\Re A(Z)$.

We obtain the following proposition which implies Corollary~\ref{cor13a}.
\begin{proposition}
	If $(M^{2n+1},\theta)$ is a strictly pseudohermitian $\CC$-convex real hypersurface in a K\"ahler manifold $(\mathcal{X}, \omega)$ with nonnegative holomorphic sectional curvature. Then $M$ has $\frac{1}{2}$-positive holomorphic sectional curvature. In particular, if $n=1$ and $M^3$ is compact, then there is a real parameter $\lambda > 0$ such that the (adapted) Riemannian metric
	\begin{align}
		g_{\lambda}:=d\theta + \lambda^{-2} \theta\otimes\theta 
	\end{align}
	has positive Riemannian scalar curvature.
\end{proposition}
\begin{proof}
	If $\mathcal{X}$ has nonnegative holomorphic sectional curvature $\tilde{K} \geqslant 0$, then from Theorem~\ref{thm:m1}, we deduce that for $|Z| =1 $,
	\begin{align*}
		K(Z) \geqslant \frac{1}{2}|H|^2 \geqslant \frac{1}{2}|A(Z)|
		\geqslant 0,
	\end{align*}
	and hence $(M,\theta)$ is $\frac12$-positivity. 
	Observe that the inequality is strict since $M$ is supposed to be strictly pseudohermitian $\CC$-convex. When $n=1$, it follows that $(M,\theta)$ is $\frac12$-positive in the sense of \cite{cao2020} and hence the corollary follows from Theorem~1.1 of \cite{cao2020}.
\end{proof}
\begin{proof}[Proof of Corollary~\ref{cor13}]
	Let $\omega = i a_{j\kba} dz^j \wedge d\zba^k$, then $\tilde{K} = 0$ and $\theta = \iota ^{\ast} \omega$. On the other hand, the squared mean curvature $|H|^2$ can be computed easily (see \cite{son2019semi}), namely, 
	\begin{align*}
		|H|^2 = \frac{1}{|\partial\vr|^2}.
	\end{align*}
	Thus the corollary follows from Theorem~\ref{thm:m1}.
\end{proof}
\begin{proof}[Proof of Corollary~\ref{cor15}] From Corollary~\ref{cor13}, we find that
	\begin{equation*}
		R = 2K(Z) \geqslant |\partial\vr|^{-2} > 0.
	\end{equation*}
	Thus, we can apply \eqref{ed} to deduce that $\lambda_1(\Box_b) \geqslant \frac{1}{2}\min_M |\partial\vr|^{-2}$ and complete the proof.
\end{proof}
\begin{remark}
	Consider the Reinhardt real hypersurface in $\CC^2$ given by
	\begin{align*}
		M_{\epsilon}: = \left\{(z,w) \in \CC^2 \colon (\log |z|)^2 + (\log |w|)^2 = \epsilon^2\right\}, \quad \epsilon > 0.
	\end{align*}
	This example (and its general dimension version) has been studied in various papers; see, e.g., \cite{reiter2019chern} and the references therein.
	There exists an unique pseudohermitian structure $\theta_{\epsilon}$ on $M_{\epsilon}$ such that $(M_{\epsilon},\theta_{\epsilon})$ is locally isomorphic to $(E,\theta_E)$ (with $n=1$), i.e., there exists a local CR diffeomorphism $\varphi$ in a neighborhood of each point $p\in M_{\epsilon}$ into $E$ such that $\theta_{\epsilon} = \varphi^{\ast}\theta_E$. Thus, local consideration on $M_{\epsilon}$ and $E$ are the same. Thus, on $(M_\epsilon,\theta_{\epsilon})$, we have $R = |A_{11}| = \frac{1}{2}$ and $M_{\epsilon}$ is $C_0$-positive for any $C_0 < \frac{1}{2}$. Note that $M_{\epsilon}$ is diffeomorphic to $\mathbb{T}^3$ and hence does \textit{not} support any positive scalar curvature Riemannian metric by an well-known theorem of Schoen-Yau \cite{schoen} and Gromov-Lawson \cite{gl}. Thus, the threshold $\frac12$ for $C_0$ in the aforementioned theorem of Cao, Chang, and Chen \cite{cao2020} is sharp: The theorem cannot be extended to the case of $C_0$-positivity with $C_0 < \frac{1}{2}$.
\end{remark}
\section{An example: Brieskorn manifolds}\label{sec:brieskorn}
The purpose of this section is to present some examples of semi-\break-isometrically immersed CR manifolds in a complex euclidean space and in a K\"ahler manifold of nonpositive holomorphic sectional curvature. They are the CR links of the well-known Brieskorn-Pham singularities, \edit{the Brieskorn manifolds,} which have been studied in the literature from many aspects. In several complex variables and CR geometry, it was studied by, e.g., Ebenfelt, Huang, and Zaitsev \cite{ehz} who showed that to certain extend the local CR geometry of a CR link \edit{(of a possibly more general isolated singularity)} determines the local complex structure of the singularity; see \cite{ehz} and the references therein. \edit{These examples show that the nonnegativity condition of $\tilde{K}$ in Corollary 1.3 is necessary. We also give an explicit formula for their Tanaka-Webster holomorphic sectional curvature.}

A Brieskorn manifold is the CR link of a Brieskorn-Pham variety defined by $p=0$ where
\begin{equation}\label{bp}
	p(z_0,z_1,\dots, z_N) := \sum_{j=0}^N z_j ^{a_j}
\end{equation}
and $a_j$'s are integers, $a_j\geqslant 2$. These manifolds were also analyzed by Tanaka in \cite{tanaka1975differential} as examples of ``normal'' CR manifolds.

The variety $V: = p^{-1}(0)$ has an isolated singularity at the origin. Put
\begin{equation*}
	d = \mathrm{lcm}\{a_j\colon j=0,2,\dots, N\}, 
	\ w_j = d / a_j.
\end{equation*}
Then $p$ is a weighted homogeneous polynomial with weights $(w_0, w_1,\dots; d)$, that is,
\begin{equation*}
	p(\lambda^{w_0}z_0, \dots , \lambda^{w_N} z_N) = \lambda^d p(z_0, \dots , z_N).
\end{equation*}
The link of $V$ at the origin is a Brieskorn manifold $M(r): = V \cap \{|z|^2 = r\}$, which is strictly pseudoconvex. In fact, if $\zeta$ is $|z|^2$ restricted to $V$, then away from the origin $\zeta$ is a strictly plurisubharmonic defining function for $M(r)$ in $V$ (cf. \cite{tanaka1975differential}).

Let
\begin{equation}\label{e:51}
	\theta: = \iota ^{\ast}\left(\frac{i}{2} \sum_{j} w_j^{-1} \left(z_j d\zba_j - \zba_j dz_j\right)\right).
\end{equation}
Then $\theta$ is a pseudohermitian structure on $M(r)$ (see \cite{tanaka1975differential}). Let $\CC^{N+1}$ be equipped with the K\"ahler metric with the symplectic form
\begin{equation}
	\omega = i \sum_{j} w_j^{-1} dz_j \wedge d\zba_j
\end{equation}
whose primitives includes 
\begin{equation*}
	\tilde{\theta}: = \frac{i}{2} \sum_{j} w_j^{-1} \left(z_j d\zba_j - \zba_j dz_j\right).
\end{equation*}
Thus, $\theta = \iota^{\ast} \tilde{\theta}$ and
\begin{equation*}
	d\theta = \iota^{\ast}\omega.
\end{equation*}
That is, $(M(r),\theta)$ is semi-isometrically immersed (as a real codimension 3 submanifold) in $(\CC^{N+1},\omega)$.

Let $T$ denote the vector field generating the $\mathbb{S}^1$-action induced by the $\CC^{\ast}$-action on $V$. Thus,
\begin{equation}\label{e:55}
	T = i\sum_{j=0}^{N} w_j\left(z_j\frac{\partial}{\partial z_j} - \zba_j\frac{\partial}{\partial \zba_j}\right).
\end{equation}
Then $T$ is the Reeb field associated to $\theta$.
Since the $\mathbb{S}^1$-action which $T$ generates is holomorphic, $\theta$ has vanishing pseudohermitian torsion \cite{tanaka1975differential,webster1978pseudo}. We obtain
\begin{proposition}\label{prop:24}
	The immersion $\iota \colon (M,\theta) \longrightarrow (V,(\iota_V^{\CC^n})
	^{\ast} \omega)$ is a semi-isometric totally umbilical CR immersion.
\end{proposition}
The proof follows from the fact that $T$ is holomorphic (via \cite{tanaka1975differential} and \cite{webster1978pseudo}) and the relation between the torsion and the second fundamental form in \eqref{e:gausstorsion}. We leave the details to the readers.
\begin{remark}
	By the well-known results of Gromov-Lawson \cite{gl} and Milnor \cite{milnor} (see also \cite{ohta}), we can find among Brieskorn manifolds many examples of pseudohermitian $\CC$-convex real hypersurfaces in K\"ahler surfaces of negative holomorphic sectional curvature admitting \textit{no} positive scalar Riemannian metrics. This shows that the nonnegativity of $\tilde{K}$ in Corollary~\ref{cor13a} is necessary.
\end{remark}

Explicit computations of the (Tanaka-Webster) scalar, Ricci, as well as the full curvature tensor have been done only for few examples, see e.g., \cite{li2011cr,chanillo2012embeddability, son2019semi}. \edit{Proposition~\ref{prop44} below gives another one: We use the Gauss equations in Proposition~\ref{prop:ge} to compute the holomorphic sectional curvature of a Brieskorn manifold.}
\begin{proposition}\label{prop44}
	The Tanaka-Webster holomorphic sectional curvature of the Brieskorn CR manifold $(M(r),\theta)$ is given by
	\begin{equation}\label{e69}
		K(Z)
		=
		\sum_{j=0}^N w_j |z_j|^2-\frac{1}{2\|\xi\|^2}\left|\sum_{k=0}^N a_k (a_k-1) z_k ^{a_k-2} (Z^k)^2\right|^2,
	\end{equation}
	where $Z = Z^k \partial_k$ belongs to $T^{(1,0)}(M)$ \edit{and has} unit length if and only if
	\begin{equation*}
		\sum_{k=0}^N a_kZ^k z_k^{a_k-1} = 0,
		\quad 
		\sum_{k=0}^N Z^k \zba_k = 0,
		\quad \text{and}\
		\sum_{k=0}^N w_k^{-1}|Z^k|^2 = 1.
	\end{equation*}
\end{proposition}
\edit{Our computation below is somewhat similar to that of Vitter \cite{vitter1974curvature} who computed curvatures of a general complex hypersurface. In fact, we shall use the following result.}
\begin{proposition}[Vitter \cite{vitter1974curvature}]
	If $Z$ and $W$ are unit $(1,0)$-vectors tangent to the non-singular locus of $V$, $Z = Z^j \partial_j$ and $W = W^j\partial_j$, then the holomorphic bisectional curvature is
	\begin{equation}\label{vitter}
		\tilde{B}(Z,W) 
		=
		-\frac{1}{\|\xi\|^2}\left|\sum_{j=0}^N \sum_{k=0}^N \frac{\partial^2 p}{\partial z_j \partial z_k} Z^jW^k\right|^2.
	\end{equation}
\end{proposition}
\begin{proof}
We reproduce Vitter's proof in \cite{vitter1974curvature} for the sake of completeness. This is also helpful since we also need a formula for the second fundamental form of the immersion $\iota_V^{\mathbb{C}^n}\colon V \to \mathbb{C}^n$ in the proof for our latter purpose.

For any $X = X^j \partial_j \in T^{(1,0)} (V)$ ($q\in V)$, we must have
\begin{equation*}
	0 = X(p) = \sum_{j=0}^N X^j \frac{\partial p}{\partial z_j}
	=
	\sum_{j=0}^N w_j^{-1} X^j \left(w_j \frac{\partial p}{\partial z_j} \right).
\end{equation*}
Therefore, if 
\begin{equation*}
	\xi = \sum_{j=1}^N w_j \left(\overline{\frac{\partial p}{\partial z_j}} \right) \frac{\partial}{\partial z_j},
\end{equation*}
then $N : = \xi/\|\xi\|$ is a $(1,0)$-normal vector of $V$.

For any $(0,1)$-vector $\Wba$ that is tangent to $V$, we have
\begin{equation*}
	\tilde{\nabla}_{\Wba} N
	=
	\tilde{\nabla}_{\Wba}\left(\frac{\xi}{\|\xi\|}\right)
	=
	\left(\Wba \left(\frac{1}{\|\xi\|}\right)\right) \xi + \left(\frac{1}{\|\xi\|}\right)\tilde{\nabla}_{\Wba} \xi.
\end{equation*}
Since $\partial/\partial z_j$ is parallel with respect to $\tilde{\nabla}$, we have
\begin{align*}
	\tilde{\nabla}_{\Wba} \xi 
	=
	\sum_{j=0}^N \left(\Wba\left(w_j \overline{\frac{\partial p}{\partial z_j}} \right)\right) \frac{\partial}{\partial z_j}
	=
	\sum_{j=0}^N \left(\sum_{k=0}^N \Wba^k w_j \left(\overline{\frac{\partial^2 p}{\partial z_j \partial z_k}}\right)\right)\frac{\partial}{\partial z_j}.
\end{align*}
Therefore, for any $\overline{Y} = \overline{Y}^j \partial _{\jba}$, we have
\begin{align*}
	\langle \tilde{\nabla}_{\Wba} \xi, \overline{Y} \rangle 
	& =
	\sum_{j=0}^N w_j^{-1}\left(\sum_{k=0}^N \Wba^k w_j \left(\overline{\frac{\partial^2 p}{\partial z_j \partial z_k}}\right)\right) \overline{Y}^j \notag \\
	& =
	\sum_{j=0}^N \sum_{k=0}^N \Wba^k\left(\overline{\frac{\partial^2 p}{\partial z_j \partial z_k}}\right) \overline{Y}^j.
\end{align*}
If $\overline{Y}$ is tangent to $V$, then by Weingarten formula, we have
\begin{align*}
	\left\langle N , \II_V^{\CC^N}(\Wba,\overline{Y}) \right\rangle
	& =
	-\left\langle \tilde{\nabla}_{\Wba} N , \overline{Y} \right\rangle \notag \\
	& = 
	-\frac{1}{\|\xi\|} \sum_{j=0}^N \sum_{k=0}^N \Wba^k\left(\overline{\frac{\partial^2 p}{\partial z_j \partial z_k}}\right) \overline{Y}^j.
\end{align*}
Thus,
\begin{equation}\label{e:65}
	\II_V^{\CC^N}(\Wba,\overline{Y}) 
	=
	-\frac{1}{\|\xi\|} \left(\sum_{j=1}^N \sum_{k=1}^N \Wba^k\left(\overline{\frac{\partial^2 p}{\partial z_j \partial z_k}}\right) \overline{Y}^j\right) \overline{N},
\end{equation}
where
\begin{equation*}
	\|\xi\|^2
	=
	\sum_{j=0}^N w_j \left|\frac{\partial p}{\partial z_j}\right|^2.
\end{equation*}
Thus, the Gauss equation for K\"ahler submanifolds, together with \eqref{e:65} yields the desired equation \eqref{vitter} immediately.
\end{proof}
\begin{proof}[Proof of Proposition~\ref{prop44}]
When $p$ is the Brieskorn-Pham polynomial given in \eqref{bp} we have
\begin{align*}
	\left\|\xi\right\|^2 
	= d\sum_{j=0}^N a_j |z_j |^{2a_j -2},
\end{align*}
and, by Vitter's formula above, the holomorphic bisectional curvature of the Brieskorn-Pham variety is 
\begin{equation*}
	\tilde{B}(Z,W)
	=
	-\frac{1}{\|\xi\|^2}\left|\sum_{k=0}^N a_k (a_k-1) z_k ^{a_k-2}Z^k W^k\right|^2.
\end{equation*}

Consider the CR immersion $\iota_M^V \colon M\to V$. As $(M,\theta)$ has vanishing torsion, we have
\begin{equation*}
	0 = \left\langle \tau Z, W \right\rangle 
	=
	- \left\langle H, \II_M^V(Z, W) \right\rangle,
\end{equation*}
and since the normal bundle $N^{(1,0)}(M)$ in $T^{(1,0)}(V)$ has 1-dimensional fibers while $H\ne 0$, we have 
\begin{equation}
	\II_M^V(Z, W) = 0 , \quad \forall\, Z,W \in T^{(1,0)}(M).
\end{equation}
If $H$ is the $(1,0)$-mean curvature vector of the immersion $\iota_M^{\CC^n}$ \edit{and $T$ is the Reeb field}, then \edit{by \cite{son2019semi}} $iT = H -\overline{H}$ and thus (by \eqref{e:55})
\begin{equation}\label{e73}
	H = - \sum_{j=0}^{N} w_j z_j \frac{\partial}{\partial z_j}.	
\end{equation}
The restriction of $H$ to \edit{$M$} is the $(1,0)$-mean curvature vector of the immersion $\iota_M^V$. Then
\begin{equation*}
	B(Z,W) = \frac{1}{2}\tilde{B}(Z,W) + \frac{1}{2}|H|^2(1+ \langle Z, \overline{W}\rangle ^2).
\end{equation*}
From this, equation \eqref{e73}, and the formula of Vitter \eqref{vitter} we obtain the formula for the holomorphic bisectional curvature:
\begin{equation*}
	B(Z,W)
	=
	\frac{1}{2}(1+\langle Z ,\Wba\rangle^2)\sum_{j=0}^N w_j |z_j|^2-\frac{1}{2\|\xi\|^2}\left|\sum_{k=0}^N a_k (a_k-1) z_k ^{a_k-2} Z^k W^k\right|^2,
\end{equation*}
which reduces to the desired formula \eqref{e69} for the holomorphic sectional curvature by setting $Z= W$.
\end{proof}

We conclude this section by pointing out a quite interesting property of the Brieskorn manifolds, the nowhere CR umbilicity. On a Levi-nondegerate CR manifold of dimension at least 5 a point $p$ is said to be \textit{CR umbilical} if the Chern-Moser-Weyl tensor (i.e., the completely tracefree part of the Tanaka-Webster curvature tensor $R_{\alpha\bba\gamma\sba}$ \cite{webster1978pseudo}) with respect to some (equivalently, to all) pseudohermitian structure $\theta$ vanishes at $p$. The Chern-Moser-Weyl tensor and CR umbilicality are important biholomorphic invariants. Compact nowhere CR umbilical CR manifolds are interesting as they admit a distinguished contact form, the principal contact form, with respect to which the Chern-Moser-Weyl tensor has unit norm. The first known examples of such manifolds were given by Webster \cite{webster2000holomorphic}. They are the generic real ellipsoids in complex space. Webster also proved that the corresponding Reeb fields are ``completely integrable.'' A principal contact form, if it exists, is useful for studying CR (or biholomorphic) equivalences of CR manifolds (or the complex domains they bound); see, e.g., \cite[Section 5]{reiter2019chern}.

The Brieskorn manifolds give a wealth of examples of nowhere CR umbilical manifolds.
\begin{corollary}\label{cor:26} Let $a = (a_0, a_1,\dots , a_N)$, $N\geqslant 3$, and let $M(r)$ be the Brieskorn manifold. If $a_k \geqslant 2$ for all $k$, then the Chern-Moser-Weyl tensor of $M$ never vanishes and hence $M$ admits a unique principal contact form.
\end{corollary}
\begin{proof}
Since $(M,\theta)$ is totally umbilical in $V$, we have from Proposition~\ref{prop:24}
\begin{equation}
	\tf \II_{M}^{\CC^N} = \tf \II_{V}^{\CC^N}\bigl|_{M}
\end{equation}
where $\tf$ denotes the trace-free part of a tensor. Equation \eqref{e:65} above shows that $\tf \II_{V}^{\CC^N}\bigl|_{M}$ never vanishes on $M$ if $a_k \geqslant 2$ for all $k$. Thus the proof follows from \cite[Proposition 3.3]{son2019semi}.
\end{proof}

\end{document}